\documentclass{amsart}
\usepackage{amssymb}
\usepackage{amsmath}
\usepackage{amsthm}
\usepackage[all]{xy}
\usepackage{enumerate}
\usepackage{hyperref}
\theoremstyle{plain}
\newtheorem{lem}{Lemma}[section]

\newtheorem{prop}[lem]{Proposition}
\newtheorem{thm}[lem]{Theorem}

\theoremstyle{definition}

\newtheorem{ex}[lem]{Example}

\newtheorem{disc}[lem]{Remark}

\newtheorem{fact}[lem]{Fact}

\newtheorem{notation}[lem]{Notation}

\newcommand{\bbz}{\mathbb{Z}}

\newcommand{\bbr}{\mathbb{R}}

\newcommand{\ve}{\varepsilon}

\renewcommand{\geq}{\geqslant}
\renewcommand{\leq}{\leqslant}

\theoremstyle{plain}
\newtheorem{theorem}[lem]{Theorem}
\newtheorem{lemma}[lem]{Lemma}
\newtheorem{corollary}[lem]{Corollary}

\theoremstyle{definition}
\newtheorem{definition}[lem]{Definition}
\newtheorem{example}[lem]{Example}

\renewcommand{\ve}[1]{\underline{#1}}
\newcommand{\monset}[1]{\left[\!\left[#1\right]\!\right]}
\newcommand{\Monset}[1]{\left[\!\!\left[#1\right]\!\!\right]}
\newcommand{\fsum}{\sum^{\text{finite}}}
\newcommand{\bbrg}{\mathbb{R}_{\geq 0}}
\newcommand{\bbrd}{\mathbb{R}_{\geq 0}^d}
\newcommand{\vep}{\varepsilon}
\newcommand{\bbrid}{\mathbb{R}_{\infty\geq 0}^d}
\newcommand{\bbri}{\mathbb{R}_{\infty\geq 0}}

\numberwithin{equation}{lem}

\begin{document}
\bibliographystyle{amsplain}

\title[Decompositions of Monomial Ideals]{Decompositions of Monomial Ideals in Real Semigroup Rings}
\author{Daniel Ingebretson}
\address{Daniel Ingebretson,
Department of Mathematics, Statistics, and Computer Science
University of Illinois at Chicago
322 Science and Engineering Offices (M/C 249)
851 S. Morgan Street
Chicago, IL 60607-7045
USA}

\email{dingeb2@uic.edu}

\author{Sean Sather-Wagstaff}
\address{Sean Sather-Wagstaff,
Department of Mathematics,
NDSU Dept \# 2750,
PO Box 6050,
Fargo, ND 58108-6050
USA}

\email{sean.sather-wagstaff@ndsu.edu}

\urladdr{http://www.ndsu.edu/pubweb/\~{}ssatherw/}

\thanks{This material was supported by North Dakota EPSCoR and 
National Science Foundation Grant EPS-0814442.
Sean Sather-Wagstaff was supported in part by a grant from the NSA}


\begin{abstract}
Irreducible decompositions of monomial ideals in polynomial rings over a field are well-understood. In this paper, we investigate decompositions in the set  of monomial ideals in the semigroup ring $A[\mathbb{R}_{\geq 0}^d]$ where $A$ is an arbitrary commutative ring with identity. We  classify the irreducible elements of this set, which we call \emph{m-irreducible}, and we   classify the elements that admit  decompositions into finite intersections of m-irreducible ideals.
\end{abstract}
\maketitle

\section{Introduction}

Throughout this paper, let $A$ be a commutative ring with identity.

\

When $A$ is a field, the polynomial ring $P=A[X_1,\ldots,X_d]$ is noetherian, so every ideal in this ring has an irreducible decomposition.
For \emph{monomial ideals}, that is, the ideals of $P$ generated by sets of monomials, these decompositions are well understood:
the non-zero irreducible monomial ideals  are 
precisely the ideals $(X_{i_1}^{e_1},\ldots,X_{i_n}^{e_n})P$ generated by ``pure powers'' of some of the variables,
and every monomial ideal in this setting decomposes as a finite intersection of irreducible monomial ideals.
Furthermore, there are good algorithms for computing these decompositions, both by hand~\cite{gao:cidmi, heinzer:pdmiii, heinzer:pdmii, liu:apdmi, roune:saidmi} 
and by computer~\cite{Cocoa, M2, singular}.\footnote{One 
of the most interesting aspects of this theory is found in its interactions with combinatorics, including
applications to graphs and simplicial 
complexes; see, e.g.,
\cite{faridi:fisc, faridi:stscm, faridi:cmpsfmi, ha:rsfmifi, ha:mieihgbn, hochster:cmrcsc, morey:eiacp, miller:adfldms, miller:admir, miller:tcmcmi, reisner:cmqpr, stanley:ubccmr, villarreal:cmg}.
Foundational material on the subject can be found in the following
texts~\cite{bruns:cmr, herzog:mi, hibi:accp, miller:cca, rogers:mid, stanley:cca, villarreal:ma}.}
Note that much of this discussion extends to the case where $P$ is replaced by a numerical semigroup ring, that is, a ring of the form $A[S]$ where $S$ is a sub-semigroup of $\bbz^d$.

When $A$ is not noetherian, some of the conclusions from the previous paragraph fail because $P$ fails to be noetherian.
However, $P$ does behave somewhat ``noetherianly'' with respect to monomial ideals. 
For instance, all monomial ideals in $P$ are finitely generated, by finite sets of monomials.
The ideals that are indecomposable with respect to intersections of monomial ideals, which we call \emph{m-irreducible}, 
are  the ideals $(X_{i_1}^{e_1},\ldots,X_{i_n}^{e_n})P$.
Each monomial ideal of $P$ admits an \emph{m-irreducible decomposition},
a decomposition into a finite intersection of m-irreducible monomial ideals,
and many of the algorithms carry over to this setting; see, e.g., \cite{rogers:mid}.

In this paper, we step even further from the noetherian setting by considering 
monomial ideals in the semigroup ring $R=A[\mathbb{R}_{\geq 0}^d]$ where $A$ is an arbitrary commutative ring with identity
and $\mathbb{R}_{\geq 0}^d = \left\{ (r_1,\ldots,r_d) \in \mathbb{R}^d \mid r_1,\ldots,r_d \geq 0 \right\}$.
This ring can be thought of as the set $A[X_1^{\mathbb{R} \geq 0}, \ldots, X_d^{\mathbb{R} \geq 0}]$ 
of all polynomials in variables $X_1,\ldots,X_d$ with coefficients in $A$ where the exponents are non-negative \emph{real} numbers.
For instance, in this ring, many of the monomial ideals are not finitely generated and many do not admit finite m-irreducible decompositions; see, e.g.,
Fact~\ref{prince}\eqref{prince2} and Example~\ref{ex110803a}.
On the other hand, every ideal admits a (possibly infinite) m-irreducible decomposition, by Proposition~\ref{prop110806a}.

Our goal is  to completely characterize the monomial ideals in $R$ that admit m-irreducible decompositions.
This is accomplished in two steps. First, we characterize the m-irreducible ideals.
This is accomplished in Theorem~\ref{eq}, which we paraphrase in the following result.

\begin{thm}\label{thm110803a}
A monomial ideal in the ring $A[\mathbb{R}_{\geq 0}^d]=A[X_1^{\mathbb{R} \geq 0}, \ldots, X_d^{\mathbb{R} \geq 0}]$ 
is m-irreducible if and only if it is generated by a set of pure powers of 
the variables $X_1,\ldots,X_d$.
\end{thm}

Our characterization of the monomial ideals in $R$ that admit m-irreducible decompositions is more technical. 
However, our intuition is straightforward, and reflects the connection between the noetherian property and
existence of decompositions: a monomial ideal in $R$ admits an m-irreducible decomposition
if and only if it is almost finitely generated. 
To make sense of this, we need to explain what we mean by ``almost finitely generated''.
We build up the general definition in steps.

First, we consider the case of monomial ideals that are ``almost principal''.
In one variable (i.e., the case $d=1$) there are exactly two kinds of non-zero monomial ideals:
given a real number $a\geq 0$ set
\begin{align*}
I_{a,0}&=(X_1^r\mid r\geq a)R=(X_1^a)R\\
I_{a,1}&=(X_1^r\mid r> a)R.
\end{align*}
These ideals are completely determined by the sets of their exponents,
corresponding exactly to open and closed rays in $\bbr_{\geq 0}$. We think of these ideals as being almost generated by $X_1^a$.
This includes the ideal that \emph{is} generated by $X_1^a$ as a special case.

In two variables (i.e., the case $d=2$) there is  more variation. First, not every ideal is almost principal; in fact, we can find ideals here that are not
almost finitely generated here. Second, the almost principal ideals come in four flavors in this setting:
given  real numbers $a,b\geq 0$ set
\begin{align*}
I_{(a,b),(0,0)}&=(X_1^rX_2^s\mid \text{$r\geq a$ and $s\geq b$})R=(X_1^aX_2^b)R\\
I_{(a,b),(1,0)}&=(X_1^rX_2^s\mid \text{$r> a$ and $s\geq b$})R\\
I_{(a,b),(0,1)}&=(X_1^rX_2^s\mid \text{$r\geq a$ and $s> b$})R\\
I_{(a,b),(1,1)}&=(X_1^rX_2^s\mid \text{$r> a$ and $s> b$})R.
\end{align*}
These ideals are completely determined by the sets of their exponent vectors,
corresponding  to combinations of open and closed rays on each axis of $\bbr^2$. We think of these ideals as being almost generated by $X_1^aX_2^b$.
In general (i.e., for arbitrary $d\geq 1$) there are $2^d$ different flavors of almost principal monomial ideals,
corresponding to the different choices of open and closed rays for the exponents of each variable; see Notation~\ref{Iae}.

In general, a monomial ideal is ``almost finitely generated'' if it is a finite sum of almost principal monomial ideals.
This definition is motivated by the fact that  a finitely generated monomial
ideal is a sum of principal monomial ideals. 
For instance, each m-irreducible monomial ideal is almost finitely generated.
In these terms, our characterization of the decomposable monomial ideals, stated next, is quite straightforward; see Theorem~\ref{decomp}:

\begin{thm}\label{thm110803b}
A monomial ideal in $R$ admits an m-irreducible decomposition
if and only if it is almost finitely generated.
\end{thm}

As to the organization of this paper, Section~\ref{sec110803a} consists of definitions and background results,
and Sections~\ref{sec110803b} and~\ref{sec110806a} are primarily concerned with the proofs of Theorems~\ref{thm110803a}
and~\ref{thm110803b}, respectively.

\section{Background and Preliminary Results}\label{sec110803a}

In this section, we lay the foundation for the proofs of our main results.
We begin  by establishing some notation for use throughout the paper.

\begin{definition}\label{R} 
Set
\begin{align*}
\bbrg&=\{r\in\bbr\mid r\geq 0\}
&\bbri&=\bbrg\cup\{\infty\}.
\end{align*}
Let $d$ be a non-negative integer, and set 
\begin{align*}
\mathbb{R}_{\geq 0}^d &=(\bbrg)^d= \left\{ (r_1,\ldots,r_d) \in \mathbb{R}^d \mid r_1,\ldots,r_d \geq 0 \right\}
\end{align*}
which is an additive semigroup, and
\begin{align*}
\bbrid &=(\bbri)^d= \left\{ (r_1,\ldots,r_d)  \mid r_1,\ldots,r_d \in\bbri \right\}
\end{align*}
We consider the semigroup ring
$$R=A[\mathbb{R}_{\geq 0}^d]$$
which we think of as the set 
of all polynomials in variables $X_1,\ldots,X_d$ with coefficients in $A$ where the exponents are non-negative real numbers.
Moreover, for $i=1,\ldots,d$ we set
$$ X_i^{\infty} = 0. $$
When the number of variables is small ($d\leq 2$) we will use variables $X,Y$ in place of $X_1,X_2$.
A \textit{monomial} in $R$ is an element of the form 
$$\ve{X}^{\ve{r}} = X_1^{r_1} \cdots X_d^{r_d} \in R$$ 
where 
$ \ve{r} = \left(r_1, \ldots, r_d\right) \in \mathbb{R}_{\geq 0}^d $ is the \textit{exponent vector} of the monomial $f$.
Multiplication of monomials in this ring is defined analogously to its classical counterpart: For all $ \ve{q}, \ve{r} \in \mathbb{R}_{\geq 0}^d $, and all $ s \in \mathbb{R} $, we write
$$ \ve{X}^{\ve{q}} \ve{X}^{\ve{r}} = \ve{X}^{\ve{q}+\ve{r}} \qquad \left( \ve{X}^{\ve{q}} \right)^{s} = \ve{X}^{\left(s \ve{q} \right)}. $$
An arbitrary element of $R$ is a  linear combination of monomials
$$ f = \fsum_{\ve{r} \in \mathbb{R}_{\geq 0}^d} a_{\ve{r}} \ve{X}^{\ve{r}}$$
with coefficients $a_{\ve r}\in A$.
A \emph{monomial ideal} of $R$ is an ideal $I=(S)R$ generated by a set $S$ of monomials in $R$.
Given a subset $ G \subseteq R $, the \textit{monomial set}  of $G$ is
$$ \monset{G} =\{\text{monomials of $R$ in $G$}\}\subseteq G. $$
\end{definition}

We next list some basic properties of monomial ideals.

\begin{fact}\label{monfact}
Let $I$ and $J$ be monomial ideals of $R$.
\begin{enumerate}[(a)]
\item\label{monfact1}
For any subset $ G \subseteq R $, we have  $ \monset{G} = G \cap \monset{R} $, by definition. 
\item\label{monfact2}
The monomial ideal $I$ is generated by its monomial set: $ I = \left( \monset{I} \right)R $.
\item\label{monfact3}
Parts~\eqref{monfact1} and~\eqref{monfact2} combine to show that $ I \subseteq J $ if and only if $ \monset{I} \subseteq \monset{J} $, and hence $ I = J $ if and only if $ \monset{I} = \monset{J} $.
\end{enumerate}
\end{fact}

The next facts follow from the $\mathbb{R}_{\geq 0}^d$-graded structure on $R$, that is, the isomorphisms 
$R\cong \oplus_{\mathbb{R}_{\geq 0}^d}A\ve X^{\ve r}\cong \oplus_{f\in\monset R}Af$.

\begin{fact}\label{fact110804a}
Let $\{I_{\lambda}\}_{\lambda\in\Lambda}$ be a set of monomial ideals of $R$.
\begin{enumerate}[(a)]
\item\label{fact110804a1}
Given monomials $ f = \ve{X}^{\ve{r}} $ and $ g = \ve{X}^{\ve{s}} $ in $R$, we have
$f| g$ if and only if $g\in(f)R$ if and only if $ s_i \geq r_i $ for all $ i $.
When these conditions are satisfied, we have $g=fh$ where $h=\ve X^{\ve s-\ve r}$.
\item\label{fact110804a2}
Given a monomial $ f \in\monset R$ and a subset $S\subseteq\monset R$, 
we have $ f \in (S)R $ if and only if $ f \in (s)R $ for some $s\in S$.
\item\label{fact110804a3}
The sum $\sum_{\lambda\in\Lambda}I_{\lambda}$ is a monomial ideal such that
$\Monset{\sum_{\lambda\in\Lambda}I_{\lambda}}=\bigcup_{\lambda\in\Lambda}\monset{I_{\lambda}}$.
\item\label{fact110804a4}
The intersection $\bigcap_{\lambda\in\Lambda}I_{\lambda}$ is a monomial ideal such that
$\Monset{\bigcap_{\lambda\in\Lambda}I_{\lambda}}=\bigcap_{\lambda\in\Lambda}\monset{I_{\lambda}}$.
\end{enumerate}
\end{fact}

Given a set $\{S_{\lambda}\}_{\lambda\in\Lambda}$ of subsets of $R$, the  equality 
$\sum_{\lambda\in\Lambda}(S_{\lambda})R=(\bigcup_{\lambda\in\Lambda}S_{\lambda})R$ is standard.
In general, one does not have such a nice description for the intersection $\bigcap_{\lambda\in\Lambda}(S_{\lambda})R$.
However, for monomial ideals in our ring $R$, the next result provides such a description.
In a sense, it says that the monomial ideals of $R$ behave like ideals in a unique factorization domain.
First, we need a definition.

\begin{definition}\label{lcm}
Let $ \ve{X}^{\ve{r_1}}, \ldots, \ve{X}^{\ve{r_k}} \in\monset R$  with $\ve{r_i}=(r_{i,1},\ldots,r_{i,d})\in\bbrd$. 
We define the \textit{least common multiple} of these monomials as $$ \underset{1 \leq i \leq k}{\operatorname{lcm}} ( \ve{X}^{\ve{r_k}} )=\ve{X}^{\ve{p}}  $$
where $ \ve{p} $ is defined componentwise by $ p_j = \underset{1 \leq i \leq k}{\max} \left\{ r_{i,j} \right\} $.
\end{definition}

\begin{lem}\label{int}
Given subsets  $ S_1,\ldots,S_k \subseteq \monset R$,  we have 
$$\textstyle \bigcap_{i=1}^k \left( S_i \right)R = \left( \underset{1 \leq i \leq k}{\textup{lcm}} (f_i) \mid\text{
$f_i \in S_i$ for $i=1,\ldots,k$} \right)R. $$
\end{lem}

\begin{proof}
Let $ L = \left( \underset{1 \leq i \leq k}{\textup{lcm}} (f_i) \mid\text{
$f_i \in S_i$ for $i=1,\ldots,k$} \right)R$, which is a monomial ideal of $R$ by definition.
Fact~\ref{fact110804a}\eqref{fact110804a4} implies that $\bigcap_{i=1}^k \left( S_i \right)R$ is also a monomial ideal
with $\Monset{\bigcap_{i=1}^k \left( S_i \right)R}=\bigcap_{i=1}^k\monset{ \left( S_i \right)R}$.
Thus, to show that $\bigcap_{i=1}^k \left( S_i \right)R= L$, we need only show that $\bigcap_{i=1}^k\monset{ \left( S_i \right)R}= \monset{L}$.

For the containment $\bigcap_{i=1}^k\monset{ \left( S_i \right)R}\subseteq \monset{L}$, let 
$\ve{X}^{\ve{t}} \in \bigcap_{i=1}^k\monset{ \left( S_i \right)R}$. 
Fact~\ref{fact110804a}\eqref{fact110804a2} implies that $ \ve{X}^{\ve{t}} $ is a  multiple of one of the generators $\ve{X}^{\ve{r_i}} \in S_i$ for  $ i = 1, \ldots, k $. 
Hence we have $ t_j \geq r_{i,j} $ for all $i$ and  $j$ by Fact~\ref{fact110804a}\eqref{fact110804a1}. 
It follows that $ t_j \geq \text{max}_{1 \leq i \leq d} \{ r_{i,j} \} = p_j $, hence $ \ve{X}^{\ve{t}} \in \left( \ve{X}^{\ve{p}} \right)R \subseteq L $.

For the reverse containment $\bigcap_{i=1}^k\monset{ \left( S_i \right)R}\supseteq \monset{L}$, let 
$\ve{X}^{\ve{t}}\in\monset L $. 
Fact~\ref{fact110804a}\eqref{fact110804a2} implies that $ \ve{X}^{\ve{t}} $ is a  multiple of one of the monomial  generators of $L$,
so there exist  $ \ve{X}^{\ve{r_i}} \in S_i $ for $i=1,\ldots,k$ such that $ \ve{X}^{\ve{t}} $ is a multiple of $ \operatorname{lcm}_{1 \leq i \leq k} ( \ve{X}^{\ve{r}_i} )$. 
Since $ \operatorname{lcm}_{1 \leq i \leq k} ( \ve{X}^{\ve{r}_i} )$ is a multiple of $\ve{X}^{\ve{r}_i}$ for each $i$, we then have
$\ve{X}^{\ve{t}} \in \bigcap_{i=1}^k\left(  \ve{X}^{\ve{r}_i}  \right)R \subseteq \bigcap_{i=1}^k \left( S_i \right)R$, as desired.
\end{proof}

\begin{lem}\label{lem110806a}
Let $G\subseteq\monset R$, and set $I=(G)R$.
Let $ \ve{X}^{\ve{b}} \in \monset I $ be given, and 
for $j=1,\ldots,d$ set
$ I_j = ( G \cup \{ X_j^{b_j} \} ) R$.
Then we have $ I = \bigcap_{j=1}^d I_j $.
\end{lem}

\begin{proof}
The containment $G\subseteq G \cup \{ X_j^{b_j}\}$ implies that $I=(G)R\subseteq(G \cup \{ X_j^{b_j}\})R=I_j$, so we have
$ I \subseteq \bigcap_{j=1}^d I_j $.
For the reverse  containment, Facts~\ref{monfact}\eqref{monfact3} and~\ref{fact110804a}\eqref{fact110804a4} imply that it suffices to show that 
$\monset I \supseteq \bigcap_{j=1}^d \monset{I_j}$.

Let $ \alpha \in \bigcap_{j=1}^d \monset{I_j} $, and suppose that $\alpha\notin I$. 
For  $ j = 1, \ldots, d $, we have $ \alpha \in I_j = ( G \cup \{ X_j^{b_j} \} ) R$. 
The condition $\alpha\notin I=(G)R$ implies that $\alpha$ is not a multiple of any element of $G$, so Fact~\ref{fact110804a}\eqref{fact110804a2} implies that
$\alpha\in(X_j^{b_j})R$ for $j=1,\ldots,d$.
In other words, we have 
$\alpha\in\bigcap_{j=1}^d(X_j^{b_j})R=(\ve X^{\ve b})R\subseteq I$
by Lemma~\ref{int}.
This contradiction establishes the lemma.
\end{proof}

The last result of this section describes the interaction between sums and intersections of monomial ideals in $R$.

\begin{lem}\label{comm}
For $t=1,\ldots,l$, let $ \{ K_{t,i_t} \}_{i_t = 1}^{m_t} $ be a collection of monomial ideals. Then the following equalities hold: 
\begin{align} 
\label{comm1}\tag{a}
\bigcap_{t=1}^l \sum_{i_t =1}^{m_t} K_{t,i_t} &= \sum_{i_1 =1}^{m_1} \sum_{i_2 = 1}^{m_2} \cdots \sum_{i_l = 1}^{m_l} \bigcap_{t=1}^l K_{t,i_t} \\
\label{comm2}\tag{b}
\sum_{t=1}^l \bigcap_{i_t =1}^{m_t} K_{t,i_t} &= \bigcap_{i_1 =1}^{m_1} \bigcap_{i_2 = 1}^{m_2} \cdots \bigcap_{i_l = 1}^{m_l} \sum_{t=1}^l K_{t,i_t}
\end{align}
\end{lem}

\begin{proof} 
\eqref{comm1} 
Note that  the left- and right-hand sides of  equation~\eqref{comm1} 
are monomial ideals of $R$ by Fact~\ref{fact110804a}\eqref{fact110804a3}--\eqref{fact110804a4}. Because of Fact~\ref{monfact}\eqref{monfact3},  equation~\eqref{comm1} follows
from the next sequence of equalities:
\begin{align*}
\textstyle\Monset{ \bigcap_{t=1}^l \sum_{i_t =1}^{m_t} K_{t, i_t} } 
& \textstyle=  \bigcap_{t=1}^l \Monset{ \sum_{i_t =1}^{m_t} K_{t, i_t} } \\
& \textstyle=   \bigcap_{t=1}^l \left[ \bigcup_{i_t =1}^{m_t} \monset{ K_{t, i_t} } \right] \\
& \textstyle=  \bigcup_{i_1 =1}^{m_1} \bigcup_{i_2 =1}^{m_2} \cdots \bigcup_{i_l =1}^{m_l} \left[ \bigcap_{t=1}^l \monset{ K_{t, i_t} } \right] \\
& \textstyle=  \bigcup_{i_1 =1}^{m_1} \bigcup_{i_2 =1}^{m_2} \cdots \bigcup_{i_l =1}^{m_l} \left[ \Monset{ \bigcap_{t=1}^l K_{t, i_t} } \right] \\
& \textstyle=  \Monset{ \sum_{i_1 =1}^{m_1} \sum_{i_2 =1}^{m_2} \cdots \sum_{i_l =1}^{m_l} \bigcap_{t=1}^l K_{t, i_t} }
\end{align*}
Here, the third equality is from the distributive law for unions and intersections, while the remaining steps are from Fact~\ref{fact110804a}\eqref{fact110804a3}--\eqref{fact110804a4}.

The verification of equation~\eqref{comm2} is similar, so we omit it.
\end{proof}

\section{M-Irreducible Ideals} \label{sec110803b}

The following notation is extremely convenient for our proofs. 
To motivate the notation, note that when  $\vep$ is $1$, we are thinking of $\vep$ as an arbitrarily small positive real number.

\begin{notation}
Let $ \varepsilon \in \mathbb{Z}_2 $.
Given $ r, \alpha \in \mathbb{R} $, 
we define
$$ r \geq_\varepsilon \alpha \quad \text{provided that} \quad \begin{cases}r \geq \alpha &\text{ if  $\varepsilon = 0$} \\ r > \alpha &\text{ if $\varepsilon = 1$.} \end{cases}  $$
Given $ s \in \bbri$, 
we define
$$ s \geq_\varepsilon \infty \quad \text{provided that} \quad s=\infty. $$
\end{notation}

Employing this new notation, we define a monomial ideal $ J_{i, \alpha, \varepsilon} $ that is generated by pure powers of the single variable $ X_i $. 
Recall our convention that $X_i^{\infty}=0$.

\begin{notation}\label{Jiae}
Given $ \alpha \in \bbri$ and $ \varepsilon \in \mathbb{Z}_2 $, we set 
$$ J_{ i, \alpha, \varepsilon } = \left( \left\{ X_i^r \mid  r \geq_\varepsilon \alpha \right\} \right)R. $$
We use the term ``pure power'' to describe a monomial of the form $X_i^r$.
\end{notation}

Given $ \alpha \in \bbrg$, we use $ \varepsilon \in \mathbb{Z}_2 $ to distinguish between two important cases. Essentially, they represent the difference between
the closed interval $[\alpha,\infty)$ in the case $\vep=0$ and the open  interval $(\alpha,\infty)$ in the case $\vep=1$.
The important difference is the existence of a minimal element in the first case, but not in the second case.
The  case $\alpha=\infty$ may seem strange, but it is quite useful.

\begin{fact}
\label{prince}
Let $ \alpha \in \bbri $ and $ \varepsilon \in \mathbb{Z}_2 $.
\begin{enumerate}[(a)]
\item\label{prince0}
$ J_{ i, \infty, \vep } = 0$. 
\item\label{prince1}
If $ \varepsilon = 0 $, then $ J_{i, \alpha, \varepsilon} = \left( X_i^{\alpha} \right) R$.
\item\label{prince2}
If $\alpha<\infty$, then  $ J_{i, \alpha, \varepsilon}$ is finitely generated if and only if $\vep=0$.
\end{enumerate}
\end{fact}

The  ideal $ J_{i, \alpha, \varepsilon} $ is  generated by pure powers of the single variable $ X_i $. 
Next, we consider the  class of ideals generated by pure powers of more than one variable. 
This notation is the first place where we see the utility of our convention $X_i^\infty=0$, since it allows us to consider all the variables simultaneously
instead of worrying about partial lists of the variables.

\begin{notation}\label{Jae}
Given  $ \ve{\alpha}= \left( \alpha_1, \ldots, \alpha_d \right) \in \bbrid$, and $ \ve{\varepsilon}= \left( \varepsilon_1, \ldots, \varepsilon_d \right) \in \mathbb{Z}_2^d $,
we set
$$ J_{ \ve{\alpha}, \ve{\varepsilon} } = \left( \left\{ X_i^{r_i} \mid i = 1, \ldots, d \text{  and  }  r_i \geq_{\varepsilon_i} \alpha_i   \right\} \right)R. $$
\end{notation}

\begin{ex}\label{ex110805a}
In the case $d=2$, using  constants $a,b\in\bbrg$ and $\vep,\vep'\in\bbz_2$,  we have eight different possibilities for $J_{ \ve{\alpha}, \ve{\varepsilon} }$:
\begin{align*}
&J_{(a,b),(0,0)}=(X^a,Y^b)R
&&J_{(a,b),(0,1)}=(X^a,Y^{b'}\mid b'>b)R
\\
&J_{(a,b),(1,0)}=(X^{a'},Y^b\mid a'>a)R\quad
&&J_{(a,b),(1,1)}=(X^{a'},Y^{b'}\mid \text{$a'>a$ and $b'>b$})R
\\
&J_{(\infty,b),(\vep,0)}=(Y^b)R
&&J_{(\infty,b),(\vep,1)}=(Y^{b'}\mid b'>b)R
\\
&J_{(a,\infty),(1,\vep')}=(X^{a'}\mid a'>a)R
&&J_{(\infty,\infty),(\vep,\vep')}=0.
\end{align*}
\end{ex}

The following connection between ideals of the form $ J_{ \ve{\alpha}, \ve{\varepsilon} } $ and those of the form $ J_{ i, \alpha, \varepsilon } $ is immediate:

\begin{fact}\label{Jaedecomp}
Given $ \ve{\alpha} \in \bbrid $ and $ \ve{\varepsilon} \in \mathbb{Z}_2^d $, we have the following equality: 
$$ J_{\ve{\alpha}, \ve{\varepsilon}} = \sum_{i=1}^d J_{i, \alpha_i, \varepsilon_i} .$$
\end{fact}

The ideals defined next are the irreducible elements of the set of monomial ideals.

\begin{definition}\label{irred}
A monomial ideal $ I \subseteq R$ is \textit{m-irreducible} (short for \textit{monomial-irreducible}) provided that 
for all monomial ideals $ J $ and $ K $ of $R$ such that $ I = J \cap K $, either $ I = J $ or $ I = K $.
\end{definition}

A straightforward induction argument establishes the following property.

\begin{fact}\label{fact110806a}
Let $I$ be an m-irreducible monomial ideal of $R$.
Given monomial ideals $I_1,\ldots,I_n$ of $R$,
if $I=\cap_{j=1}^nI_j$, then $I=I_j$ for some $j$.
\end{fact}

Our first main result, which we prove next, is the fact that the m-irreducible monomial ideals of $R$ are exactly the
$J_{ \ve{\alpha}, \ve{\varepsilon} }$; it contains Theorem~\ref{thm110803a} from the introduction.
Notice that it includes the case $J_{ \ve{\infty}, \ve{\varepsilon} } =0$ where $\ve\infty=(\infty,\ldots,\infty)$.

\begin{theorem}\label{eq}
Let $ I \subseteq R $ be a monomial ideal. Then the following are equivalent:
\begin{enumerate}[\rm(i)]
\item \label{eq1}
$ I $ is generated by pure powers of a subset of the variables $ X_1, \ldots, X_d $;
\item\label{eq2} 
there exist $ \ve{\alpha} \in \bbrid$ and $ \ve{\varepsilon} \in \mathbb{Z}_2^d $ such that $ I = J_{ \ve{\alpha}, \ve{\varepsilon} } $; and
\item \label{eq3}
$ I $ is m-irreducible. 
\end{enumerate}
\end{theorem}

\begin{proof}
$\eqref{eq1}\implies\eqref{eq2}$
Assume that $I$ is generated by pure powers of a subset of  the variables. It follows that there are (possibly empty)
sets $S_1,\ldots,S_d\subseteq\bbrg$ such that
$I=\sum_{i=1}^d(\{X_i^{z_i}\mid z_i\in S_i\})R$. 
For $i=1,\ldots,d$ set $\alpha_i = \inf{S_i} \in\bbri$, and let $ \ve{\alpha} =(\alpha_1,\ldots,\alpha_d)\in\bbrid$.
(Here we assume that $\inf\emptyset = \infty$.)
Furthermore, for $i=1,\ldots,d$ set 
$$\vep_i = \begin{cases}
0&\text{if $ \alpha_i \in S_i $} \\
1&\text{if $ \alpha_i \notin S_i $} \end{cases}
$$
and let $ \ve{\vep} =(\vep_1,\ldots,\vep_d)\in\bbz^d_2$.
It is straightforward to show that $I=J_{ \ve{\alpha}, \ve{\varepsilon} }$.

$\eqref{eq2}\implies\eqref{eq3}$
Let $ \ve{\alpha} \in \bbrid$ and $ \ve{\varepsilon} \in \mathbb{Z}_2^d $ be given, and suppose by way of contradiction that $ J_{ \ve{\alpha}, \ve{\varepsilon} } $ is not m-irreducible. 
By definition, there exist monomial ideals $ J$ and $ K $ such that $ J_{ \ve{\alpha}, \ve{\varepsilon} } = J \cap K $, with $ J_{ \ve{\alpha}, \ve{\varepsilon} } \neq J $
and $ J_{ \ve{\alpha}, \ve{\varepsilon} } \neq K $. 
It follows that $ J_{ \ve{\alpha}, \ve{\varepsilon} }\subsetneq J$ and $ J_{ \ve{\alpha}, \ve{\varepsilon} }\subsetneq K$,
so Fact~\ref{monfact}\eqref{monfact3} provides 
monomials $ \ve{X}^{\ve{q}} \in J \setminus J_{ \ve{\alpha}, \ve{\varepsilon} } $ and  $ \ve{X}^{\ve{r}} \in K \setminus J_{ \ve{\alpha}, \ve{\varepsilon} } $. 
If there exist $ q_i $ or $ r_i $ such that $ r_i \geq_{\varepsilon_i} \alpha_i$ or $ q_i \geq_{\varepsilon_i} \alpha_i$, then  Fact~\ref{fact110804a}\eqref{fact110804a1}
implies that $ \ve{X}^{\ve{q}} \in \left( X_i^{\alpha_i} \right) \subseteq J_{ \ve{\alpha}, \ve{\varepsilon} } $,  a contradition. 

We conclude  that for all $ i $, we have $ q_i <  \alpha_i $ and $ r_i < \alpha_i $, so $ p_i := \max \{ q_i, r_i \} < \alpha_i $. 
By Fact~\ref{fact110804a}\eqref{fact110804a1}, this implies that $ \ve{X}^{\ve{p}} = \operatorname{lcm}( \ve{X}^{\ve{q}}, \ve{X}^{\ve{r}} )$ is not a  multiple of 
any generator of $ J_{ \ve{\alpha}, \ve{\varepsilon} } $, hence $ \ve{X}^{\ve{p}} \notin J_{ \ve{\alpha}, \ve{\varepsilon} } $ by Fact~\ref{fact110804a}\eqref{fact110804a2}. 
However, Lemma~\ref{int} implies that $ \ve{X}^{\ve{p}} \in J \cap K =J_{ \ve{\alpha}, \ve{\varepsilon} }$,  a contradiction.

$\eqref{eq3}\implies\eqref{eq1}$
Assume that $I$ is m-irreducible, and let $G\subseteq\monset R$ be a generating set for $I$.
Suppose by way of contradiction that $I$ is not generated by pure powers of some of the variables $ X_1, \ldots, X_d $. 
Then there is a  monomial $ \ve{X}^{\ve{b}} \in G $ that is not a  multiple of any pure power $ X_i^a \in G $.

For $j=1,\ldots,d$ we consider  the monomial ideals 
$ I_j = ( G \cup \{ X_j^{b_j} \} ) R$.

Claim: $ I \neq I_j $ for  $ j =1, \ldots, d $.
Suppose by way of contradiction that there exists an index $ k $ such that $ I = I_k $,
that is, such that $ ( G )R = ( G \cup \{ X_k^{b_k} \} ) R$. Fact~\ref{monfact}\eqref{monfact3} implies that $X_k^{b_k}$ is a multiple of some $g\in G$.
By Fact~\ref{fact110804a}\eqref{fact110804a1}, we conclude  that $g=X_k^a$ for some $a\leq b_k$, so we have
$\ve X^{\ve b}\in(X_k^{b_k})R\subseteq (X_k^a)R$.
That is, the monomial $\ve X^{\ve b}$ is a multiple of the pure power $X_k^a=g\in G$, a contradiction.
This establishes the claim.

Lemma~\ref{lem110806a} implies that $ I = \bigcap_{j=1}^d I_j $.
Thus, the claim conspires with Fact~\ref{fact110806a} to contradict the fact that $I$ is m-irreducible.
\end{proof}

We explicitly document a special case of Theorem~\ref{eq} for use in the sequel.

\begin{corollary}\label{cor}
Each ideal $ J_{ i, \alpha, \varepsilon } $ is m-irreducible. 
\end{corollary}

\begin{proof} 
This is the special case $\ve\alpha=(\infty,\ldots,\infty,\alpha,\infty,\ldots,\infty)$ of Theorem~\ref{eq}.
\end{proof}

\section{Ideals Admitting Finite M-Irreducible Decompositions}\label{sec110806a}

We now turn our attention to the task of characterizing the monomial ideals of $R$ that admit  decompositions into
finite intersections of m-irreducible ideals.

\begin{definition}\label{red}
Let $ I \subseteq R $ be a monomial ideal. An \textit{m-irreducible decomposition} of $ I $ is a decomposition
$ I = \bigcap_{\lambda \in \Lambda} I_{\lambda} $ where each $I_{\lambda}$ is an m-irreducible monomial ideal of $R$.
If the index set $\Lambda$ is finite, we say that $ I = \bigcap_{\lambda \in \Lambda} I_{\lambda} $ is a \emph{finite m-irreducible decomposition}. 
\end{definition}

Our second main result shows that the monomial ideals of $R$ that admit finite m-irreducible decompositions
are precisely the finite sums of ideals of the next form. 

\begin{notation}\label{Iae}
Let $ \ve{\alpha} \in \bbrid$ and $ \ve{\varepsilon} \in \mathbb{Z}_2^d $ be given, and set
$$ I_{ \ve{\alpha}, \ve{\varepsilon} } = \left( \left\{ \ve{X}^{\ve{r}} \mid i = 1, \ldots, d \text{  and  }  r_i \geq_{\varepsilon_i} \alpha_i   \right\} \right)R. $$
\end{notation}

\begin{ex}\label{ex110806a}
With the zero-vector $\ve 0=(0,\ldots,0)$ we have
$I_{\ve\alpha,\ve 0}=(\ve X^{\ve\alpha})R$.
If $\alpha_i=\infty$ for any $i$, then $I_{\ve\alpha,\ve\vep}=0$.
\end{ex}

As a first step, we show next that each ideal of the form $ I_{\ve{\alpha}, \ve{\varepsilon}}$ has a finite m-irreducible decomposition.

\begin{lem}\label{Idecomp}
Given $ \ve{\alpha} \in \bbrid$ and $ \ve{\varepsilon} \in \mathbb{Z}_2^d $, we have
$ I_{\ve{\alpha}, \ve{\varepsilon}} = \bigcap_{i=1}^d J_{i, \alpha_i, \varepsilon_i} $.
\end{lem}

\begin{proof}
If $\alpha_{i}=\infty$ for some $i$, then we have $J_{i, \alpha_{i}, \varepsilon_{i}} =0$, so 
$ \bigcap_{i=1}^d J_{i, \alpha_{i}, \varepsilon_{i}} =0$;
thus the desired equality follows from Example~\ref{ex110806a} in this case.
Assume now that $\alpha_{i}\neq\infty$ for all $i$.
In the following computation, the second equality is from Lemma~\ref{int}:
\begin{align*} \bigcap_{i=1}^d J_{i, \alpha_i, \varepsilon_i} 
&= \bigcap_{i=1}^d \left( \{ X_i^{r_i} \mid r_i \geq_{\varepsilon_i} \alpha_i\} \right)R  \\ 
&=  \left( \underset{1 \leq i \leq d}{\textup{lcm}} \left\{  X_i^{r_i} \right\} \mid r_i \geq_{\varepsilon_i} \alpha_i   \right)R  \\ 
& = \left( \left\{ X_1^{r_1} X_2^{r_2} \cdots X_d^{r_d} \mid r_i \geq_{\varepsilon_i} \alpha_i    \right\} \right)R \\ 
& = \left( \left\{ \ve{X}^{\ve{r}} \mid r_i \geq_{\varepsilon_i} \alpha_i   \right\} \right)R \\ 
& = I_{ \ve{\alpha}, \ve{\varepsilon} } 
\end{align*}
The first, fourth, and fifth equalities are by definition, and the third equality is straightforward.
\end{proof}

\begin{disc}\label{disc110806a}
It is worth noting that the decomposition from Lemma~\ref{Idecomp} may be redundant, in the sense that some of the ideals
in the intersection may be removed without affecting the intersection: if
$\alpha_i=0$ and $\vep_i=0$, then $J_{i, \alpha_i, \varepsilon_i} =(X_i^0)R=R$.
\end{disc}

Lemma~\ref{Idecomp} not only provides a decomposition for the ideal $ I_{\ve{\alpha}, \ve{\varepsilon} } $, 
but also gives a first indication of how the ideals $ J_{i, \alpha_i, \varepsilon_i} $  behave under intersections. 
The next lemmas partially extend this. The essential point of the first lemma is the fact that a finite intersection of real intervals
of the form $[\alpha_i,\infty)$ and $(\alpha_j,\infty)$ is a real interval of the form $[\alpha,\infty)$ or $(\alpha,\infty)$.

\begin{lem}\label{lem110806b}
Let $i,b\in\bbz$ be given such that $1\leq i\leq d$ and $b\geq 0$.
Given $\alpha_{1},\ldots,\alpha_b\in\bbri$ and $\vep_1,\ldots,\vep_b\in\bbz_2$, there exist $\beta\in\bbri$ and $\delta\in\bbz_2$ 
such that $\bigcap_{t=1}^bJ_{i, \alpha_{t}, \varepsilon_{t}}=J_{i, \beta,\delta}$. Specifically, we have
$$\beta=\begin{cases} \max\{\alpha_{1},\ldots,\alpha_b\}&\text{if $b\geq 1$} \\ 0 &\text{if $b=0$.}\end{cases}$$
\end{lem}

\begin{proof}
If $b=0$, then we have 
$$\bigcap_{t=1}^bJ_{i, \alpha_{t}, \varepsilon_{t}}=\bigcap_{t=1}^0J_{i, \alpha_{t}, \varepsilon_{t}}=R=J_{i, 0,0}$$
as claimed. Thus, we assume for the remainder of the proof that $b\geq 1$.
Since $\bigcap_{t=1}^bJ_{i, \alpha_{t}, \varepsilon_{t}}\subseteq J_{i, \alpha_{j}, \varepsilon_{j}}$ for each $j$,
it suffices to find an index $j$ such that $\alpha_j=\max\{\alpha_{1},\ldots,\alpha_b\}$ and
$J_{i, \alpha_{j}, \varepsilon_{j}}\subseteq J_{i, \alpha_{t}, \varepsilon_{t}}$ for all $t$.
If $\alpha_{j}=\infty$ for some $j$, then we have $J_{i, \alpha_{j}, \varepsilon_{j}} =0$, and we are done.
Thus, we assume for the remainder of the proof that $\alpha_{j}\neq\infty$ for all $j$.

Choose $k$  such that $\alpha_k=\max\{\alpha_{1},\ldots,\alpha_b\}$.
If there is an index $j$ such that $\alpha_j=\alpha_k$ and $\vep_j=1$, then
we have $J_{i, \alpha_{j}, \varepsilon_{j}}\subseteq J_{i, \alpha_{t}, \varepsilon_{t}}$ for all $t$
since $J_{i, \alpha_{j}, \varepsilon_{j}}$ is generated by monomials of the form $X_i^\alpha$ where 
$\alpha>\alpha_j\geq\alpha_t$.

So, we assume that for every index $j$ such that $\alpha_j=\alpha_k$ we have $\vep_j=0$.
In this case, we have $J_{i, \alpha_{k}, \varepsilon_{k}}\subseteq J_{i, \alpha_{t}, \varepsilon_{t}}$ for all $t$, as follows.
If $\alpha_k=\alpha_t$, then $\vep_t=0$, so we have 
$J_{i, \alpha_{k}, \varepsilon_{k}}=(X_i^{\alpha_k})R=(X_i^{\alpha_t})R= J_{i, \alpha_{t}, \varepsilon_{t}}$.
On the other hand, if
$\alpha_k\neq\alpha_t$, then $\alpha_k>\alpha_t$ and hence
$J_{i, \alpha_{k}, \varepsilon_{k}}=(X_i^{\alpha_k})R\subseteq J_{i, \alpha_{t}, \varepsilon_{t}}$.
\end{proof}

\begin{lemma}\label{intk}
Let $ k $ be a positive integer. 
For $t=1,\ldots,k$ let $i_t\in\{1,\ldots,d\}$ be given, and fix $\alpha_{t}\in\bbri$ and $\varepsilon_{t}\in\bbz_2$.
Then the intersection
$ \bigcap_{t=1}^k J_{i_t, \alpha_{t}, \varepsilon_{t}} $ is a monomial ideal of the form $ I_{ \ve{\beta}, \ve{\delta} } $ 
for some $ \ve{\beta} \in \mathbb{R}_{\geq 0}^d $ and $ \ve{\delta} \in \mathbb{Z}_2^d $.
\end{lemma}

\begin{proof}
If $\alpha_{i_j}=\infty$ for some $j$, then we have $J_{i_j, \alpha_{j}, \varepsilon_{j}} =0$, so 
$ \bigcap_{t=1}^k J_{i_t, \alpha_{t}, \varepsilon_{t}} =0$,
and the desired equality follows from Example~\ref{ex110806a} in this case.
Thus, we assume for the remainder of the proof that $\alpha_{i_j}\neq\infty$ for all $j$.

Reorder the $i_t$'s if necessary to obtain the first equality in the next sequence where $0\leq b_1\leq b_2\leq\cdots\leq b_d \leq b_{d+1}=k+1$:
\begin{align*}
\bigcap_{t=1}^k J_{i_t, \alpha_{t}, \varepsilon_{t}} 
&= \bigcap_{i=1}^d \bigcap_{t=b_i}^{b_{i+1}-1} J_{i, \alpha_{t}, \varepsilon_{t}}
=\bigcap_{i=1}^dJ_{i, \beta_i, \delta_i}= I_{ \ve{\beta}, \ve{\delta} }. 
\end{align*}
The second step is from Lemma~\ref{lem110806b}, and the third step is from Lemma~\ref{Idecomp}.
\end{proof}

We demonstrate the algorithm from the proof of Lemma~\ref{intk} in the next example.

\begin{example}\label{ex1}
Let $ d = 2 $. We show how to write the ideal 
$$
I=J_{1,2,1} \cap J_{2, \frac{3}{2}, 0} \cap J_{1, \frac{5}{3}, 0} \cap J_{2, 1, 1} = \left[ J_{1, 2, 1} \cap J_{1, \frac{5}{3}, 0} \right] \cap \left[ J_{2, \frac{3}{2}, 0} \cap J_{2, 1, 1} \right] 
$$
in the form $ I_{\ve{\beta}, \ve{\delta}} $.
Define $ \ve{\beta} \in \mathbb{R}_{\geq 0}^2 $ and $ \ve{\delta} \in \mathbb{Z}_2^2 $ by $ \beta_1 = \text{max} \{ 2, \frac{5}{3} \} = 2 $, $ \delta_1 = 1 $,  $ \beta_2 = \text{max} \{ \frac{3}{2}, 1 \} = \frac{3}{2} $, and $ \delta_2 = 0 $. Then we have
$I= J_{1, \beta_1, \delta_1} \cap J_{2, \beta_2, \delta_2} 
 =  I_{\ve{\beta}, \ve{\delta}}$.
\end{example}

Lemmas~\ref{lem110806b} and~\ref{intk} show how to simplify an arbitrary intersection of  ideals of the form $ J_{i, \alpha_i, \varepsilon_i} $. 
The next lemmas are proved similarly and show how to simplify an arbitrary sum of these ideals.

\begin{lem}\label{lem110806c}
Let $i,b\in\bbz$ be given such that $1\leq i\leq d$ and $b\geq 0$.
Given $\alpha_{1},\ldots,\alpha_b\in\bbri$ and $\vep_1,\ldots,\vep_b\in\bbz_2$, there exist $\beta\in\bbri$ and $\delta\in\bbz_2$ 
such that $\sum_{t=1}^bJ_{i, \alpha_{t}, \varepsilon_{t}}=J_{i, \beta,\delta}$. Specifically, we have
$$\beta=\begin{cases} \min\{\alpha_{1},\ldots,\alpha_b\}&\text{if $b\geq 1$} \\ \infty &\text{if $b=0$.}\end{cases}$$
\end{lem}

\begin{lemma}\label{sumk}
Let $ k $ be a positive integer. 
For $t=1,\ldots,k$ let $i_t\in\{1,\ldots,d\}$ be given, and fix $\alpha_{i_t}\in\bbri$ and $\varepsilon_{t}\in\bbz_2$.
Then $ \sum_{t=1}^k J_{i_t, \alpha_{t}, \varepsilon_{t}} $ is a monomial ideal of the form $ J_{ \ve{\beta}, \ve{\delta} } $ for some 
$ \ve{\beta} \in \bbrid $ and $ \ve{\delta} \in \mathbb{Z}_2^d $.
\end{lemma}

The next example is included to shed some light on Lemma~\ref{sumk}.

\begin{example}\label{ex2}
Let $ d = 2 $. We show how to write the ideal 
$$
I=J_{1,\frac{9}{8},1} + J_{2, \frac{11}{2}, 0} + J_{1, \frac{14}{3}, 0} + J_{2, 3, 1} = \left[ J_{1, \frac{9}{8}, 1} + J_{1, \frac{14}{3}, 0} \right] + \left[ J_{2, \frac{11}{2}, 0} + J_{2, 3, 1} \right] 
$$
in the form $J_{\ve{\beta}, \ve{\delta}}$.
Define $ \ve{\beta} \in \mathbb{R}_{\geq 0}^2 $ and $ \ve{\delta} \in \mathbb{Z}_2^2 $
by $ \beta_1 = \text{min} \{ \frac{9}{8}, \frac{14}{3} \} = \frac{9}{8} $, $ \delta_1 = 0 $, $ \beta_2 = \text{min} \{ \frac{11}{2}, 3 \} = 3$, and $ \delta_2 = 0 $.
Then we have
$I=  J_{1, \beta_1, \delta_1} + J_{2, \beta_2, \delta_2} = J_{\ve{\beta}, \ve{\delta}}$.
\end{example}

We are now in a position to prove the main result of this section, which contains Theorem~\ref{thm110803b} from the introduction.

\begin{theorem}\label{decomp}
A monomial ideal $I\subseteq R$ has a finite m-irreducible decomposition if and only if it can be expressed as a finite sum of ideals of the form $ I_{\ve{\alpha}, \ve{\varepsilon}} $.
\end{theorem}

\begin{proof}
$ \implies $: Assume that $ I $ has a finite m-irreducible decomposition. 
Theorem~\ref{eq} explains the first equality in the next sequence:
\begin{align*}
I 
&= \bigcap_{t=1}^k J_{\ve{\alpha}_t, \ve{\varepsilon}_t} \\
&= \bigcap_{t=1}^k \sum_{i_t=1}^d J_{i_t, \alpha_{t,i_t}, \varepsilon_{t,i_t}}\\
&= \sum_{i_1=1}^d \sum_{i_2=1}^d \cdots \sum_{i_k=1}^d \bigcap_{t=1}^k J_{i_t, \alpha_{t,i_t}, \varepsilon_{t,i_t}} \\
&= \sum_{i_1=1}^d \sum_{i_2}^d \cdots \sum_{i_k=1}^d I_{\ve{\beta}_i, \ve{\delta}_i}
\end{align*}
The remaining equalities are from Fact~\ref{Jaedecomp}, Lemma~\ref{comm}\eqref{comm1}, and Lemma~\ref{intk}.
Thus, the ideal $I$ is a sum of the desired form.

$ \impliedby $: Assume that $I$ is a finite sum of ideals of the form $ I_{\ve{\alpha}, \ve{\varepsilon}}$:
\begin{align*}
I 
&= \sum_{t=1}^k I_{\ve{\alpha}_t, \ve{\varepsilon}_t}\\
&= \sum_{t=1}^k \bigcap_{i_t=1}^d J_{i_t, \alpha_{t,i_t}, \varepsilon_{t,i_t}}\\
&= \bigcap_{i_1=1}^d \bigcap_{i_2=1}^d \cdots \bigcap_{i_k=1}^d \sum_{t=1}^k J_{i_t, \alpha_{t,i_t}, \varepsilon_{t,i_t}}\\
&= \bigcap_{i_1=1}^d \bigcap_{i_2=1}^d \cdots \bigcap_{i_k=1}^dJ_{\ve{\beta}_i, \ve{\delta}_i}.
\end{align*}
The second, third, and fourth steps in this sequence are from Lemmas~\ref{Idecomp}, \ref{comm}\eqref{comm2}, and~\ref{sumk},
respectively.
This expresses $ I $ as a finite intersection of ideals of the form $ J_{\ve{\beta}, \ve{\delta}} $,
so Theorem~\ref{eq} implies that $ I $ has a finite m-irreducible decomposition.
\end{proof}

The next example exhibits a monomial ideal in $R$ that does not admit a finite m-irreducible decomposition.
The discussion of the case $d=1$ in the introduction shows every monomial ideal in this case is m-irreducible.
Thus, our example must have $d\geq 2$.
The essential point for the example is that the graph of the line $y=1-x$ that defines the exponent vectors of the generators of $I$ is not 
a ``descending staircase'' which is the form required for an ideal to be a finite sum of 
ideals of the form $ I_{\ve{\alpha}, \ve{\varepsilon}} $. 
Note that any ideal defined by a similar curve (e.g., any curve of the form $y=f(x)$ where $f(x)$ is a
non-negative, strictly decreasing, continuous function on the non-negative interval $[a,b]$) will have the same property.

\begin{example}\label{ex110803a}
Set $d=2$.
We show that the  ideal $ I= \left( \{ X^r Y^{1-r} \mid 0 \leq r \leq 1 \} \right)R $ does not admit a finite m-irreducible decomposition.
By Theorem~\ref{decomp}, it suffices to show that $I$ cannot be written as a finite sum of 
ideals of the form $ I_{\ve{\alpha}, \ve{\varepsilon}} $.

Suppose by way of contradiction that $ I = \sum_{i=1}^k I_{\ve{\alpha}_i, \ve{\varepsilon}_i} $. 
If there are indices $i$ and $j$ such that $I_{\ve{\alpha}_i, \ve{\varepsilon}_i} \subseteq I_{\ve{\alpha}_j, \ve{\varepsilon}_j} $,
then we may remove $I_{\ve{\alpha}_j, \ve{\varepsilon}_j} $ from the list of ideals without changing the sum. 
Repeat this process for each pair of 
indices $i$, $j$ such that$I_{\ve{\alpha}_i, \ve{\varepsilon}_i} \subseteq I_{\ve{\alpha}_j, \ve{\varepsilon}_j} $
to reduce to the case where no such containments occur in the sum;
since the list of ideals is finite, this process terminates in  finitely many steps.

We claim that for each real number $r$ such that $0< r< 1$, there is an index $i$
such that
$I_{\ve{\alpha}_i, \ve{\varepsilon}_i}=(X^{r},Y^{1-r})R=I_{(r,1-r),(0,0)}$. 
(This will imply that the index set $\{1,\ldots,k\}$ cannot be finite, a contradiction.)
The monomial $X^rY^{1-r}$ is in 
$\monset{I}=\bigcup_{i=1}^k \Monset{I_{\ve{\alpha}_i, \ve{\varepsilon}_i} }$ by
Fact~\ref{fact110804a}\eqref{fact110804a3}. 
It follows that $X^{r}Y^{1-r}\in \Monset{I_{\ve{\alpha}_i, \ve{\varepsilon}_i} }$ for some $i$.
Since $X^{r'}Y^{1-r},X^{r}Y^{r''}\notin I$ for all $r',r''\in\bbrg$ such that
$r'<r$ and $r''<1-r$, it follows readily that $I_{\ve{\alpha}_i, \ve{\varepsilon}_i}=(X^{r},Y^{1-r})R=I_{(r,1-r),(0,0)}$. 
\end{example}

Our final result  shows that every monomial ideal of $R$
has a possibly infinite m-irreducible decomposition and
can be written as a possibly infinite sum
of ideals of the form $ I_{\ve{\alpha}, \ve{\varepsilon}} $. 
\footnote{Note that the case $S=\emptyset$ is covered by the convention that the empty sum of ideals is the zero ideal;
the case $\monset R\setminus\monset I=\emptyset$ is similarly covered since the empty intersection contains the empty product
which is the unit ideal.}

\begin{prop}\label{prop110806a}
Let $I$ be a monomial ideal with monomial generating set $S$. Then there are equalities
$$I=\sum_{\ve{X}^{\ve{r}}\in S}I_{\ve{r},\ve{0}}=\bigcap_{\ve{X}^{\ve{r}}\notin I}J_{\ve{r},\ve{1}}$$
where $\ve{0}=(0,\ldots,0)$ and $\ve{1}=(1,\ldots,1)$.
\end{prop}

\begin{proof}
The first equality is straightforward, using Example~\ref{ex110806a}:
$$I=(S)R=\sum_{\ve{X}^{\ve{r}}\in S}(\ve{X}^{\ve{r}})R=\sum_{\ve{X}^{\ve{r}}\in S}I_{\ve{r},\ve{0}}.$$

Fact~\ref{fact110804a}\eqref{fact110804a4} implies that the ideal
$\bigcap_{\ve{X}^{\ve{r}}\notin I}J_{\ve{r},\ve{1}}$
is a monomial ideal of $R$, so Fact~\ref{monfact}\eqref{monfact3} implies that we need only show that
$\monset I=\bigcap_{\ve{X}^{\ve{r}}\notin I}\Monset{J_{\ve{r},\ve{1}}}$.

For the containment $\monset I\subseteq\bigcap_{\ve{X}^{\ve{r}}\notin I}\Monset{J_{\ve{r},\ve{1}}}$
let $\ve X^{\ve \alpha}\in\monset I$ and $\ve X^{\ve r}\notin I$; we need to show that $\ve X^{\ve \alpha}\in J_{\ve{r},\ve{1}}$,
that is, that $\alpha_i>r_i$ for some $i$.
Suppose by way of contradiction that $\alpha_i\leq r_i$ for all $i$.
Then $\ve X^{\ve r}\in(\ve X^{\ve \alpha})R\subseteq I$, a contradiction.

For the reverse containment $\monset I\supseteq\bigcap_{\ve{X}^{\ve{r}}\notin I}\Monset{J_{\ve{r},\ve{1}}}$,
we show that $\monset R\setminus\monset I\subseteq\monset R\setminus\bigcap_{\ve{X}^{\ve{r}}\notin I}\Monset{J_{\ve{r},\ve{1}}}$.
Let $\ve X^{\ve \beta}\in\monset R\setminus\monset I$. 
It follows that $\ve X^{\ve \beta}$ is in the index set for the intersection $\bigcap_{\ve{X}^{\ve{r}}\notin I}\Monset{J_{\ve{r},\ve{1}}}$.
Since $\beta_i\not>\beta_i$ for all $i$,
we have $\ve X^{\ve \beta}\notin J_{\ve\beta,\ve 1}$, so
$$\ve X^{\ve \beta}\in\monset R\setminus\Monset{J_{\ve\beta,\ve 1}}\subseteq\bigcup_{\ve{X}^{\ve{r}}\notin I}\left(\monset R\setminus\Monset{J_{\ve{r},\ve{1}}}\right)
=\monset R\setminus\bigcap_{\ve{X}^{\ve{r}}\notin I}\Monset{J_{\ve{r},\ve{1}}}
$$
as desired.
\end{proof}

\section*{Acknowledgments}
We are grateful to the anonymous referee for her/his thoughtful comments.

%\bibliography{../+new}
\providecommand{\bysame}{\leavevmode\hbox to3em{\hrulefill}\thinspace}
\providecommand{\MR}{\relax\ifhmode\unskip\space\fi MR }
\providecommand{\MRhref}[2]{%
  \href{http://www.ams.org/mathscinet-getitem?mr=#1}{#2}
}
\providecommand{\href}[2]{#2}

\end{document}